\documentclass{amsart}
\usepackage{graphicx} 
\usepackage{tikz-cd}
\usepackage[utf8]{inputenc}
\usepackage{amsmath,amssymb,amsthm}

\def\Q{{\mathbb Q}}
\def\Z{{\mathbb Z}}
\def\C{{\mathbb C}}

\newtheorem{thm}{Theorem}

\newtheorem{lem}[thm]{Lemma}
\newtheorem{coro}[thm]{Corollary}
\newtheorem{prop}[thm]{Proposition}

\title{Galois Cohomology of Reductive Groups over Function Fields over $\C((t))$}
\author{Dylon Chow}

\begin{document}

\begin{abstract}
Over a global field (number field or function field of a curve over a finite field), theorems for the Galois cohomology of algebraic groups have long been known. For $F$ the function field of a curve over the formal series field $\C((t))$, under an additional assumption on the curve, we establish an explicit description of the localization map in Galois cohomology. This is achieved by extending the theory of $B(F,G)$ due to Kottwitz. 
\end{abstract}

\maketitle

\section{Introduction}

Let $k=\C((t))$ and let $F$ be the function field of a smooth, projective, geometrically integral curve $X$ over $k$. Let $G$ be a connected reductive group over $F$. This article discusses, under certain conditions on $X$, a pointed set $B(G)$ attached to $G$ and the field $F$, as well as a pointed set $B(G)$ attached to $G$ and the completions of $F$ at closed points of $X$. These results are in parallel with those established by Kottwitz \cite{kottwitz2014} for local and global fields. 

We let $X^{(1)}$ denote the set of closed points on the $k$-curve $X$. This is in bijection with the set $V_F$ of places (discrete valuations of rank one) of the field $F=k(X)$ that are trivial on $k$. From our results on $B(G)$, we obtain results on the first Galois cohomology set $H^1(F,G)$ of $G$ and the localization map \[H^1(F,G) \to \prod_{v \in V_F} H^1(F_v,G).\] This extends results of Borovoi \cite{borovoi}, Kottwitz \cite{kottwitz2014}, and Borovoi-Kaletha \cite{borkal}, who worked over global fields.

The analogous localization map over global fields has been much studied in connection with the local-global principle. A number of recent articles address Galois cohomology over what are called fields of arithmetic type. These are fields that are not necessarily global fields but that share many properties of global fields. (See \cite{parimala} for more on this topic.) The field $F$ is a field of arithmetic type. In particular, $F$ has cohomological dimension $2$ and presents analogies with global fields of positive characteristic, namely, with function fields in one variable over a finite field.

The main local result is Theorem \ref{local}, which gives a bijection between the Galois coinvariants of Borovoi's fundamental group with the set of basic elements in $B(F,G)$. The main global result is Theorem \ref{global}, which gives a bijection between the set of basic elements in $B(F,G)$ and the Galois coinvariants of Borovoi's fundamental group with the tensor product of a certain subset of the divisor group of $X$. For global function fields, this result is due to Kottwitz \cite[Proposition 13.1, Corollary 15.2]{kottwitz2014}, whose proofs we adapt to the context of function fields over $\C((t))$. For global function fields there is a natural topological generator of the absolute Galois group of a finite field, namely the Frobenius automorphism. The absolute Galois group of $\C((t))$ is also isomorphic to $\widehat{\Z}$, the profinite completion of $\Z$, but the bijections we obtain are not canonical because they depend on a choice of a canonical generator of the absolute Galois group.

For other results related to those obtained in this article, see the paper \cite{ctharari}, which establishes arithmetic duality theorems for the Galois cohomology of tori and finite Galois modules in the context of function fields over $\C((t))$.

Section 2 of this article discusses the local theory, while section 3 discusses the global theory.

\section{Local Theory}


\subsection{Review of Tate-Nakayama triples} Throughout this article, if $\Gamma$ is a finite group and $M$ is a $\Gamma$-module, $H^r(\Gamma,M)$ denotes the $r$th Tate cohomology group of $\Gamma$ with coefficients in $M$. We recall the notion of a Tate-Nakayama triple due to Kottwitz \cite{kottwitz2014}. Let $G$ be a finite group. Let $X$ and $A$ be $G$-modules, and let $\alpha \in H^2(G,\mathrm{Hom}(X,A))$ We say that $(X,A,\alpha)$ is a \textit{weak Tate-Nakayama triple for $G$} if the following condition holds for every subgroup $G'$ of $G$: 
\begin{itemize}
    \item For all $r \in \Z$, cup product with $\mathrm{Res}_{G/G'}(\alpha)$ induces isomorphisms \[H^r(G',X) \to H^{r+2}(G',A).\]
\end{itemize}
We say that $(X,A,\alpha)$ is \textit{rigid} if 
\begin{itemize}
    \item $H^1(G',\mathrm{Hom}(X,A))$ is trivial for every subgroup $G'$ of $G$.
\end{itemize}
A \textit{Tate-Nakayama triple} is a weak Tate-Nakayama triple that is also rigid.

\subsection{Notation} 

We consider a function field over $k=\C((t))$. We let $F$ be a completion of this field with respect to a discrete valuation on the field that is trivial on $\C((t))$. The residue field of $F$ is a finite extension of $\C((t))$. We let $\overline{k}=\cup_n \C((t^{1/n}))$ and choose an algebraic closure $\overline{F}$ of $F$. We let $\Gamma(\overline{F}/F)$ denote its Galois group. If $E$ is a field, $\mathrm{Br}(E)$ denotes the Brauer group of $E$.

\subsection{Review of local class field theory}

Although $F$ is not a local field, the main results of local class field theory hold for $F$. This is summarized as follows, see for instance, \cite[Section 6.3]{GS}. First we note that $k=\C((t))$ is quasi-finite, i.e., it is perfect and its Galois group is isomorphic to the profinite completion $\widehat{\Z}$ of $\Z$. We fix a topological generator $\sigma=\sigma_k$ of $\Gamma(\overline{k}/k)$, or equivalently, a fixed isomorphism $\varphi_k \colon \Xi(\Gamma_k) \to \Q/\Z$ where $\Xi(\Gamma_k)$ is the character group $\text{Hom}_{cts}(\Gamma_k,\Q/\Z)$ of $\Gamma_k$. The relation between $\sigma$ and $\varphi_k$ is that $\varphi_k(\chi)=\chi(\sigma)$ for all $\chi \in \Xi(\Gamma_k)$. A finite extension $l$ of a quasi-finite field $k$ is again quasi-finite with generator $\sigma_l = \sigma_k^{[l:k]}$.

We refer to \cite{ADT} for the definition of a class formation that is used in this article.

\begin{prop}
    The pair $(\Gamma(\overline{F}/F),\overline{F}^\times)$ is a class formation.
\end{prop}

\begin{proof}
    There is a split exact sequence \cite[Corollary 6.3.5]{GS} \[0 \to \mathrm{Br}(\kappa) \to \mathrm{Br}(K) \to \mathrm{Hom}_{\mathrm{cts}}(\Gamma(\kappa_s/\kappa),\Z) \to 0.\] Since $\C((t))$ is of dimension $\le 1$, $\kappa$ is also of dimension 1, and so $\mathrm{Br}(\kappa)=0$. The commutativity of the diagram follows from \cite[Proposition 6.3.7]{GS}. This proposition is stated only for finite residue field, but the proof works in this setting, since we have fixed a topological generator of $\Gamma(k^s/k)$.
\end{proof}

Let $K$ be a finite Galois extension of $F$. The isomorphism \[\mathrm{inv}_F \colon \mathrm{Br}(F) \cong \Q/\Z\] obtained from the class formation $(\Gamma(\overline{F}/F),\overline{F}^\times)$ induces an isomorphism \[\mathrm{inv}_{K/F} \colon H^2(\Gamma(K/F),K^\times) \cong \frac{1}{n}\Z/\Z\] (See Theorem 2.1 in Milne's notes on class field theory). We define the \textit{(local) fundamental class} $\alpha(K/F)$ of $K/F$ to be the element $\alpha(K/F)$ of $H^2(\Gamma(K/F),K^\times)$ such that \[\mathrm{inv}_{K/F}(\alpha(K/F))=\frac{1}{[K:F]} \ \mathrm{mod} \ \Z.\]

One can rephrase this theorem as saying that we get a Tate-Nakayama triple $(X,A,\alpha)$ for $\Gamma(K/F)$ by taking 

\begin{itemize}
    \item $X$ to be $\Z$, with $\Gamma(K/F)$ acting trivially;
    \item $A$ to be $K^\times$, with the natural $\Gamma(K/F)$-action, and
    \item $\alpha$ to be the fundamental class in $H^2(\Gamma(K/F),K^\times)$.
\end{itemize}

\subsection{Algebraic 1-cocycles} We choose an extension \[1 \to K^\times \to W(K/F) \to \Gamma(K/F) \to 1\] that corresponds to the local fundamental class $\alpha(K/F)$. Notice that $W(K/F)$ is a local Weil group for the extension $K/F$. If $M$ is a $\Gamma(K/F)$-module, the extension defines an action of $W(K/F)$ on $M$, with $K^\times$ acting trivially.

We review the definition of an algebraic 1-cocycle, following \cite{kottwitz2014}. Let $G$ be a linear algebraic group over $F$. The Galois group $\Gamma=\Gamma(K/F)$ acts on $G(K)$. We consider the set $Z^1(W(F/K),G(K))$ of 1-cocycles of $W(F/K)$ in $G(K)$. An \textit{algebraic $1$-cocycle} of $W(F/K)$ in $G(K)$ is a pair $(\nu,x)$ consisting of 

\begin{itemize}
    \item a homomorphism $\nu \colon D \to G$ over $K$, and 

    \item an abstract $1$-cocycle $x$ of $W(K/F)$ in $G(K)$,
\end{itemize}

satisfying the following two compatibilities:

\begin{itemize}
    \item $x_d = \nu(d)$ for all $d \in D(E)$,

    \item $\text{Int}(x_w) \circ \sigma(\nu)=\nu$ whenever $w \in \mathcal{E}$ maps to $\sigma \in \Gamma(E/F)$.
\end{itemize}


We consider the set $Z^1(W(K/F),G(K))$ of abstract 1-cocycles of $W(K/F)$ in $G(K)$ and the set $Z^1_{\mathrm{alg}}(W(K/F),G(K))$ of algebraic $1$-cocycles of $W(K/F)$ in $G(K)$. We write $H^1_{\mathrm{alg}}(W(K/F),G(K))$ for the pointed set obtained as the quotient of $Z^1_{\mathrm{alg}}(W(K/F),G(K))$ by the action of $G(K)$.

Let $T$ be an $F$-torus that splits over $K$. Kottwitz has defined an isomorphism \[\kappa_T \colon H^1_{\mathrm{alg}}(W(K/F),T(K)) \cong X_*(T)_{\Gamma(K/F)},\] uniquely characterized by certain compatibility properties (see \cite{kottwitz2014}, Lemma 4.1).

Let $\Lambda_G$ denote Borovoi's algebraic fundamental group (see \cite{borovoi} for the definition and basic properties). 

\begin{prop}
    There is a unique map of natural transformations \[\kappa_G \colon H^1_{\mathrm{alg}}(W(K/F),G(K)) \to (\Lambda_G)_{\Gamma(K/F)}\] that agrees with $\kappa_T$ when $G=T$ is an $F$-torus split by $K$.
\end{prop}

\begin{proof}
See \cite{kottwitz2014}, Prop. 9.1.
\end{proof}

\subsection{Local $B(F,G)$} We consider local fields $L \supset K \supset F$ with both $L$ and $K$ finite Galois over $F$. One has a natural map \[H^1_{\mathrm{alg}}(W(K/F),G(K)) \to H^1_{\mathrm{alg}}(W(L/F),G(L))\] (see expression (8.32) in \cite{kottwitz2014}, page 35) and it is an isomorphism when $G$ is a torus split by $K$ \cite[Lemma 8.1]{kottwitz2014}. For any linear algebraic $F$-group $G$ we define a pointed set $B(F,G)$ to be the direct limit \[B(F,G)=\varinjlim H^1_{\mathrm{alg}}(W(K/F),G(K)),\] the direct limit being taken over the set of finite Galois extensions $K$ of $F$ in $\overline{F}$. For $L \supset K$, the transition map is the inflation map \[H^1_{alg}(W(K/F),G(K)) \to H^1_{alg}(W(L/F), G(L))\] defined above.


\subsection{Basic elements (local)}

We review the definition a subset $B(F,G)_{\mathrm{bsc}} \subset B(F,G)$  of basic elements, following Kottwitz. For fields $L \supset K \supset F$ with both $L$ and $K$ finite Galois over $F$ and contained in $\overline{F}$, let $\Z_K$ and $\Z_L$ be copies of $\Z$ attached to $K$ and $L$. Let $p \colon \Z_K \to \Z_L$ be defined by multiplication by $[L:K]$. We define $\mathbb{D}_F$ to be the protorus over $F$ whose character group is \[X^*(\mathbb{D}_F)=\varinjlim \Z_K = \Q\] with the transition maps $p \colon \Z_K \to \Z_L$. Thus $\mathbb{D}_F = \varprojlim  \mathbb{G}_{m,K}$, the limit being taken over the directed set of finite Galois extensions of $F$ in $\overline{F}$. The map $(\nu,x) \mapsto \nu$ induces a well-defined map \[H^1_{\mathrm{alg}}(W(K/F),G(K)) \to (\mathrm{Hom}_K(\mathbb{G}_m,G)/\mathrm{Int}(G(K))^{\Gamma(K/F)}\] which is called a Newton map. These maps fit together to give a Newton map \[B(F,G) \to [\mathrm{Hom}_{\overline{F}}(\mathbb{D}_F,G)/G(\overline{F})]^{\Gamma}.\] The image of $b \in B(F,G)$ under the Newton map is called the Newton point of $b$. 

The natural maps $H^1(\Gamma(K/F),G(K)) \hookrightarrow H^1_{\mathrm{alg}}(W(K/F),G(K))$ fit together to give an injective map \[H^1(F,G) \hookrightarrow B(F,G),\] whose image is the kernel of the Newton map.




We denote by $Z(G)$ the center of $G$. The inclusion $Z(G) \hookrightarrow G$ induces an injection \[\mathrm{Hom}_{\overline{F}}(\mathbb{D}_F,Z(G)) \hookrightarrow [\mathrm{Hom}_{\overline{F}}(\mathbb{D}_F,G)/G(\overline{F})]^\Gamma.\] We say that $b \in B(F,G)$ is \textit{basic} if its Newton point lies in the image of the map above. We write $B(F,G)_{\mathrm{bsc}}$ for the set of basic elements in $B(F,G)$. The set $B(F,G)_{bsc}$ contains the image of $H^1(F,G)$ under the inclusion (1.1).

\subsection{Main theorem (local)}

\begin{lem}
Let $G$ be a connected reductive group over $F$ and let $G^{\mathrm{ss}}$ denote the derived subgroup of $G$. Then $G$ contains an elliptic maximal $F$-torus, i.e., a maximal $F$-torus $T$ such that $T \cap G^{\mathrm{ss}}$ is an anisotropic $F$-torus. 
\end{lem}

\begin{proof}
Since every maximal torus of $G^{\mathrm{ss}}$ is contained in a maximal torus of $G$, we may assume that $G$ is semisimple. We use the proof of \cite[Theorem 6.21]{platonov}. The lemma in this reference states the result for non-archimedean local fields, but the proof works in our setting. We have to adapt the proof because we need to know that there exists a cyclic extension of $K$ of degree $n$ for every integer $n \ge 2$. But this follows from the fact that the residue field of $F$ is quasi-finite.
\end{proof}

The main result of this section is the following theorem. Let $\Gamma=\Gamma(\overline{F}/F)$.

\begin{thm} \label{local}
    The map $\kappa_G$ restricts to a bijection \[\kappa_G \colon B(F,G)_{\mathrm{bsc}} \to (\Lambda_G)_{\Gamma}.\]
\end{thm}

\begin{proof} The proof is the same as that of \cite[Prop. 13.1]{kottwitz2014}, using the lemma above.
\end{proof}

\begin{coro}
    The map $\kappa_G$ restricts to a bijection $H^1(F,G) \to (\Lambda_G)_{\Gamma,\mathrm{tors}}$.
\end{coro}

We note that it is possible to give another proof of the corollary without using $B(G)$, for example by using hypercohomology as in \cite{borovoi}. One obtains a map \[H^1(K,G) \cong (\Lambda_G)_{\Gamma,\mathrm{tors}}\] as the composition of an abelianization map with an isomorphism coming from local Tate-Nakayama duality (see \cite[Prop. 4.1(i)]{borovoi}). However, since the theory of $B(G)$ over local and global fields has applications to automorphic forms, we hope that developing a theory of $B(G)$ will yield interesting applications.

\section{Global Theory}

\subsection{Notation} 

In this section, we let $F$ be a function field over $k=\C((t))$ and let $X$ be a geometrically connected projective smooth curve over $k$ with function field $F$. We fix a topological generator $\sigma$ of the absolute Galois group of $\C((t))$. The set of closed points of $X$ will be denoted by $X^{(1)}$ (thus $X^{(1)}$ omits only the generic point of $X$). We are interested in the family of completions $F_v$ of $F$ relative to discrete valuations of rank 1 that are trivial on $k=\C((t))$, that is to say the valuations corresponding to the closed points of the $k$-curve $X$. We let $R_v$ denote the ring of integers in $K_v$. For such a valuation $v$, the residue field $k(v)$ is a quasi-finite field of degree $\mathrm{deg}(v)$ over $k=\C((t))$ with $\sigma^{\mathrm{deg}(v)}$ as the chosen generator of $\Gamma(\overline{k(v)}/k(v))$. It is of the form $\C((s))$ with $s=t^{1/n}$ for suitable $n$. 

We let $\mathrm{Jac}_X$ denote the Jacobian variety of $X$. From now on, we make the following assumption on the curve $X$.

\begin{itemize}
    \item For every finite extension $k'$ of $k$, the cohomology group $H^1(\Gamma_{k'},\mathrm{Jac}_X(\overline{k}))=0$.
\end{itemize}

In particular, the Brauer group $\text{Br}(X)$ is assumed to be trivial \cite[Proposition A.13]{ADT}. We make this assumption in order to obtain a global class formation, as reviewed below.

We define the \textit{group of ideles} $I_F$ of $F$ to be the subgroup of $\prod_{v \in X^{(1)}}F_v^\times$ comprising those elements $a=(a_v)$ such that $a_v \in R_v^\times$ for all but finitely many $v$. The quotient of $I_F$ by $F^\times$ (embedded diagonally) is the \textit{idele class group} $C_F$ of $F$. We set $C=\varinjlim C_K$ (limit over all finite extensions $K$ of $F$, $K \subset \overline{F}$).

Let $K/F$ be a finite Galois extension with Galois group $\Gamma(K/F)$. Then $\Gamma(K/F)$ operates on the idele group $I_K$ of $K$, on the multiplicative group $K^\times$ of $K$, and on the idele class group $C_K$ of $K$. Corresponding to the thee $\Gamma(K/F)$-modules $I_K, C_K$, and $K^\times$ there are three $\Gamma(K/F)$-modules of much simpler structure.

We let $\Z[V_K]$ denote the divisor class group of the function field $K$, i.e., the free abelian group having one basis element $x_P$ for each closed point $P$ of $X$, and we let $\Gamma(K/F)$ operate on $\Z[V_K]$ according to its effect on the primes, namely, $\lambda(x_P)=x_{\lambda(P)}$. The module that corresponds to $C_K$ is $\Z$ with the trivial action. The natural homomorphism $\Z[V_K] \to \Z$ defined by $x_P \mapsto 1$ for all $P$ is surjective, and we define $\Z[V_K]_0$ to be the kernel of this homomorphism. In other words, $\Z[V_K]_0$ is the kernel of the homomorphism $\Z[V_K] \to \Z$ defined by $\sum_{v \in V_K}n_v v \mapsto \sum_{v \in V_K}n_v$.

We define $\mathbb{D}_{K/F}$ to be the $F$-group of multiplicative type whose character group is $\Z[V_K]_0$.

\subsection{Review of global class field theory over function fields}
Write $a_v$ for the image of an element $a$ of $\mathrm{Br}(F)$ in $\mathrm{Br}(F_v)$, and define \[\mathrm{inv}_F \colon \mathrm{Br}(F) \to \Q/\Z\] to be $a \mapsto \sum_v \mathrm{inv}_v(a_v)$, where $\text{inv}_v$ is $\mathrm{inv}_{F_v}$. 

\begin{lem}
There is an exact sequence \[0 \to H^1(\Gamma_k,\mathrm{Jac}_X(\overline{k})) \to \mathrm{Br}(F) \to \bigoplus_{v \in X^{(1)}}\mathrm{Br}(F_v) \to \Q/\Z \to 0.\] 
\end{lem}

\begin{proof}
    This is \cite[Theorem A.7, page 131]{ADT}.
\end{proof}

\begin{lem}
If $H^1(\Gamma_{k'},\mathrm{Jac}_X(\overline{k}))=0$ for all finite extensions $k'$ of $k$, then it is possible to define on $(\Gamma_F,C)$ a natural structure of a class formation.
\end{lem}

\begin{proof}
    This is \cite[Corollary A.8, page 132]{ADT}.
\end{proof}

\begin{coro}
 Suppose that $H^1(\Gamma_{k'},\mathrm{Jac}_X(\overline{k}))=0$ for all finite extensions $k'$ of $k$. For every finite Galois extension $K$ of $F$, there is an isomorphism \[\mathrm{inv}_{K/F} \colon H^2(\Gamma(K/F),C_K) \to \frac{1}{n}\Z/\Z.\] 
\end{coro}

The cohomology class $\alpha \in H^2(\Gamma(K/F),C_K)$ such that $\mathrm{inv}_{K/F}(\alpha)=1/n$ is called the \textit{global fundamental class}.

\subsection{Cohomology of tori}

In this section, we extend the results of \cite{Tate1966} to our context. We compare the two exact sequences of $\Gamma(K/F)$-modules \[1 \to K^\times \to I_K \to C_K \to 1\] and \[1 \to \mathbb{Z}[V_K]_0 \to \mathbb{Z}[V_K] \to \mathbb{Z} \to 1.\] 


\begin{thm}
There are elements 
\begin{align*}
    \alpha_1 &\in H^2(\Gamma(K/F), \mathrm{Hom}(\mathbb{Z},C_K)), \\
    \alpha_2 &\in H^2(\Gamma(K/F),\mathrm{Hom}(\Z[V_E],I_K)), \\
    \alpha_3 &\in H^2(\Gamma(K/F),\mathrm{Hom}(\Z[V_K]_0,K^\times))
\end{align*} such that cup product by these cohomology classes induces isomorphisms 
\begin{align*}
    H^r(\Gamma(K/F),\Z) &\to H^{r+2}(\Gamma(K/F),C_K), \\
    H^r(\Gamma(K/F),\Z[V_K]) &\to H^{r+2}(\Gamma(K/F),I_K), \\
    H^r(\Gamma(K/F),\Z[V_K]_0) &\to H^{r+2}(\Gamma(K/F),K^\times)
\end{align*} for all $r$, $-\infty<r<\infty$.
\end{thm} 

\begin{proof}
The class $\alpha_1$ is the global fundamental class of $K/F$. For $\alpha_2$ and $\alpha_3$, see \cite{Tate1966}.
\end{proof}


\begin{prop}
    The triple $(\Z[V_K]_0,K^\times,\alpha_3)$ is a Tate-Nakayama triple.
\end{prop}

\begin{proof}
    See \cite{kottwitz2014}, section 6.3.
\end{proof}







\subsection{$B(F,G)$ (global)}

We choose an extension \[1 \to \mathrm{Hom}(\Z[V_K]_0,K^\times) \to \mathcal{E}(K/F) \to \Gamma(K/F) \to 1\] whose associated cohomology class is $\alpha_3$. Using this extension, we get (see \cite{kottwitz2014}), for each linear algebraic group over $F$, a pointed set $H^1_{\mathrm{alg}}(\mathcal{E}(K/F),G(K))$.

Given a larger finite Galois extension $L \supset K$, there is a natural $\Gamma(L/F)$-map $p \colon \Z[V_K]_0 \to \Z[V_L]_0$ defined as follows. The value of $p$ on the basis element $v \in V_K$ is defined to be \[p(v)=\sum_{w \lvert v}[L_w :K_v]w.\] Note that since $[L:K]=\sum_{w \lvert v}[L_w:K_v]$, $p$ is well-defined.


Using $p$, one forms an inflation map \[H^1_{\mathrm{alg}}(\mathcal{E}(K/F),G(K)) \to H^1_{\mathrm{alg}}(\mathcal{E}(L/F),G(L))\] (see \cite{kottwitz2014}, expression (8.32) on page 42). Using these inflation maps as transition morphisms, we form a pointed set $B(F,G)$ as the colimit of $H^1_{\mathrm{alg}}(\mathcal{E}(K/F),G(K))$, with $K$ varying over the directed set of finite Galois extensions of $F$ in some fixed separable closure $\overline{F}$ of $F$: \[B(F,G)=\varinjlim H^1_{\mathrm{alg}}(\mathcal{E}(K/F),G(K)).\]

For a place $u$ of $F$ there is a localization map \[B(F,G) \to B(F_u,G).\] The definition of this map is identical to that in \cite{kottwitz2014}, expression (10.7) on page 49.

\begin{prop}
Let $G$ be a connected reductive $F$-group, the total localization map \[B(F,G) \to \prod_{u \in V_F}B(F_u,G)\] takes values in $\bigoplus_{u \in V_F}B(F_u,G)$. That is, the components of any element in the image of the total localization map are trivial for all but finitely many $u \in V_F$.   
\end{prop}

\begin{proof}
    The same proof of Corollary 14.3 in \cite{kottwitz2014} works, after replacing Lang's Theorem by Steinberg's vanishing theorem for connected linear algebraic groups over fields of dimension $\le 1$ \cite[Theorem 1.9]{Steinberg1965}.
\end{proof}

We choose a finite Galois extension $K$ of $F$ in $\overline{F}$ such that $\Gamma(\overline{F}/K)$ acts trivially on $\Lambda_G$. Put \[A(F,G) = (\Lambda_G \otimes \Z[V_K]_0)_{\Gamma(K/F)}.\] There is a natural map (see \cite[Section 11.5]{kottwitz2014}) \[\kappa_G \colon B(F,G) \to A(F,G).\] We have the composed map \[\Z[V_K]_0 \hookrightarrow \Z[V_K] \to \Z[V_u],\] where $V_u$ is the set of places of $K$ lying over $u$. Tensoring with $\Lambda_G$ and forming $\Gamma(K/F)$-coinvariants, we obtain the composed map \[(\Lambda_G \otimes \Z[V_K]_0)_{\Gamma(K/F)} \to (\Lambda_G \otimes \Z[V_K])_{\Gamma(K/F)} \to (\Lambda_G \otimes \Z[V_u])_{\Gamma(K/F)} \cong (\Lambda_G)_{\Gamma(K_v/F_u)}.\] Then we get a map \[(\Lambda_G \otimes \Z[V_K]_0)_{\Gamma(K/F)} \to (\Lambda_G)_{\Gamma(K_v/F_u)}.\] By taking the colimit, we obtain a localization map \[A(F,G) \to A(F_u,G).\]

\subsection{Basic elements (global)}

There is a \textit{(global) Newton map} \[B(F,G) \to [\mathrm{Hom}_{\overline{F}}(\mathbb{D}_F,G)/G(\overline{F})]^\Gamma,\] where $\Gamma=\mathrm{Gal}(\overline{F}/F)$ and $\mathbb{D}_F=\varprojlim \mathbb{D}_{K/F}$, the limit being taken over the directed set of finite Galois extensions $K$ of $F$ in $\overline{F}$. The kernel of the Newton map is the image of $H^1(F,G)$ under the natural inclusion.

Inside the target of the Newton map is the subset $\text{Hom}_F(\mathbb{D}_F,Z(G))$. The preimage of $\mathrm{Hom}_F(\mathbb{D}_F,Z(G))$ is by definition the set $B(F,G)_{\mathrm{bsc}}$ of \textit{basic elements} of $B(F,G)$. The image of $H^1(F,G)$ in $B(F,G)$ is contained in $B(F,G)_{\mathrm{bsc}}$. 

\subsection{Main theorem (global)}

We make some preliminary remarks.

\begin{lem}
    Let $H$ be a semisimple simply connected algebraic group over $F$. Then $H^1(F,H)=0$.
\end{lem}

\begin{proof}
One can reduce to the case where $H$ is absolutely almost simple. For inner forms of type $A_n$, this is due to Merkurjev and Suslin and for the other classical groups it is a theorem of Bayer and Parimala \cite{bayerparimala}. For exceptional groups, it follows from a theorem of Gille \cite{gille} as well as index-period properties of $C_2$ fields (see \cite{PY}, Theorem on page 194). 
\end{proof}

\begin{lem}
    Let $S$ be a finite nonempty set of places of $F$. There exists a maximal $F$-torus $T$ in $G$ such that $T$ is elliptic over $F_u$ for every $u \in S$.
\end{lem}

\begin{proof}
We choose an elliptic torus $M^v$ for each $v \in S$.  We follow the proof of Section 7.1 Corollary 3 of \cite{platonov}; see also the proof of Theorem 2 of \cite{Prarap}. Let $\mathcal{T}$ denote the variety of maximal tori of $G$. For any valuation $v$ of $K$, the map $\phi \colon G(K_v) \to \mathcal{T}(K_v)$ given by $g \mapsto g \cdot M^v$, where $g \cdot M_v = g^{-1} M^v g$, is open in the $v$-adic topology. Thus, for any $v \in S$, the orbit set $G(F_v) \cdot M^v$ is open in $\mathcal{T}(K_v)$. Since $\mathcal{T}$ is rational and smooth, it satisfies weak approximation, and so there exists a maximal torus \[T \in \mathcal{T}(F) \cap \prod_{v \in S} (G(F_v) \cdot M^v).\] Now $G(F_v) \cdot M^v$ consists of $F_v$-tori that are conjugate to $M^v$ and so are in particular elliptic. This completes the proof. 
\end{proof}

The main results of this section are the following theorem and corollaries.

\begin{thm} \label{global}
    The map $\kappa_G \colon B(F,G) \to A(F,G)$ restricts to a bijection \[\kappa_G \colon B(F,G)_{\mathrm{bsc}} \to A(F,G) = (\Lambda_G \otimes \Z[V_K]_0)_{\Gamma(K/F)}.\]
\end{thm}

\begin{proof} 
One only needs to modify the proof of \cite[Proposition 15.1]{kottwitz2014} by using the two lemmas above.
\end{proof}

\begin{coro}
    The diagram 
    \[\begin{tikzcd}
        B(F,G)_{\mathrm{bsc}} \arrow{r} \arrow{d}{\kappa_G} 
        & \bigoplus_{u \in V_F}B(F_u,G)_{\mathrm{bsc}} \arrow{d}{\kappa_G} \\
        A(F,G) \arrow{r} & \bigoplus_{u \in V_F}A(F_v,G)
    \end{tikzcd}\]
    commutes and the vertical arrows are bijections. 
\end{coro}

\begin{proof}
    It remains to prove the compatibility of the local and global maps $\kappa_G$ with the localization maps. This follows from \cite[Lemma 11.7]{kottwitz2014}.
\end{proof}

\begin{coro}
An element in $\bigoplus_{u \in V_F}B(F_u,G)_{\mathrm{bsc}}$ lies in the image of the localization map \[B(F,G)_{\mathrm{bsc}} \to \bigoplus_{u \in V_F} B(F_u,G)_{\mathrm{bsc}}\] if and only if its image under \[\bigoplus_{u \in V_F}B(F_u,G)_{\mathrm{bsc}} \to \bigoplus_{u \in V_F} (\Lambda_G)_{\Gamma(K_v/F_u)} \to (\Lambda_G)_{\Gamma(K/F)}\] is trivial.
\end{coro}

\begin{coro}
    The map $B(F,G)_{\mathrm{bsc}} \to A(F,G)$ restricts to a bijection \[H^1(F,G) \to (\Lambda_G \otimes \Z[V_K]_0)_{\Gamma(K/F),\mathrm{tors}}=A(F,G)_{\mathrm{tors}}.\]
\end{coro}

The bijection gives $H^1(F,G)$ the structure of an abelian group.

\subsection{Concluding Remarks}

The approach of Borovoi-Kaletha \cite{borkal} can probably be modified to prove the corollary above. In this paper the assumption made on the curve $X$ was crucial, in order to get global Tate-Nakayama triples. We cannot say anything about what happens if the assumption does not hold, in particular, if the Brauer group of the $k$-curve $X$ is not trivial.

\bibliographystyle{plain}
\bibliography{biblio}

\end{document}